\documentclass{amsart}
\usepackage{amssymb,amsmath,amscd,xy,graphicx,textcomp,mathtools}
\usepackage{epsf} 
\usepackage{titlepic}
\usepackage{hyperref}
\usepackage{tabulary}
\usepackage{booktabs}

\hypersetup{
  colorlinks   = true, %Colours links instead of ugly boxes
  urlcolor     = blue, %Colour for external hyperlinks
  linkcolor    = blue, %Colour of internal links
  citecolor   = blue %Colour of citations
}

\newtheorem{theorem}{Theorem}[section]
\newtheorem{lemma}[theorem]{Lemma}
\newtheorem{corollary}[theorem]{Corollary}
\newtheorem{definition}[theorem]{Definition}

\newtheorem{example}[theorem]{Example}
\newtheorem{proposition}[theorem]{Proposition}

\xyoption{arrow}

\xyoption{matrix}

\def\pr{\mathrm{pr}}
\def\gr{\mathrm{gr}}
\def\BZ{\mathrm{BZ}}
\def\conv{\mathrm{conv}}
\def\cone{\mathrm{cone}}
\def\Proj{\mathrm{Proj}}

\def\gr{\mathrm{gr}}

\def\GL{\mathrm{GL}}

\def\SL{\mathrm{SL}}

\def\sl{\mathfrak{sl}}

\def\C{\mathbb{C}}
\def\R{\mathbb{R}}
\def\Z{\mathbb{Z}}
\def\N{\mathbb{N}}

\def\tree{\mathcal{T}}

\def\Proj{\mathrm{Proj} \,}

\begin{document}

\title[Conformal blocks, BZ triangles and group-based models]{Conformal blocks, Berenstein-Zelevinsky triangles and group-based models}

\author{Kaie Kubjas, Christopher Manon}
\date{\today}

\begin{abstract}
Work of Buczy{\'n}ska, Wi{\'s}niewski, Sturmfels and Xu, and the second author has linked the group-based phylogenetic statistical model associated with the group $\Z/2\Z$ with the Wess-Zumino-Witten (WZW) model of conformal field theory associated to $\SL_2(\C)$. In this article we explain how this connection can be generalized to establish a relationship between the phylogenetic statistical model for the cyclic group $\Z/m\Z$ and the WZW model for the special linear group $\SL_m(\C).$ We use this relationship to also show how a combinatorial device from representation theory, the Berenstein-Zelevinsky (BZ) triangles, correspond to elements in the affine semigroup algebra of the $\Z/3\Z$ phylogenetic statistical model.  

\end{abstract}

\maketitle

\tableofcontents

\section{Introduction}

In this paper we study an algebraic relationship between two branches of applied mathematics, phylogenetic algebraic geometry and the Wess-Zumino-Witten model of conformal field theory.

Phylogenetics utilizes algebraic varieties associated with a statistical model in order to construct phylogenies between empirically observed taxa. A special class of these models, the group-based models, derive some of their parameters from a finite abelian group $A$. There is one such model for each graph $\Gamma$ and finite abelian group $A.$  We consider the coordinate algebra of the variety  $M_{\Gamma}(A)$ associated with this model. 

The Wess-Zumino-Witten model of conformal field theory constructs a vector space $V_{C, \vec{p}}(\vec{\lambda}, L)$ for each stable marked projective curve $C, \vec{p},$ and tuple of representation theoretic data $\vec{\lambda}, L$ for a simple Lie algebra $\mathfrak{g}$.   Summing over the possible values of the data $\vec{\lambda}, L$ produces an infinite dimensional vector space $V_{C, \vec{p}}(G)$ which carries the structure of a commutative algebra.  This algebra is the total coordinate ring of the moduli space $\mathcal{M}_{C, \vec{p}}(G)$ of principal $G$ bundles on the curve $C$ with quasi-parabolic structure at marked points $\vec{p} \subset C.$

A connection between these objects was first made in the paper \cite{SX10} of Sturmfels and Xu, where they construct a flat degeneration of $V_{C, \vec{p}}(\SL_2(\C))$ to the coordinate ring of each $M_{\tree}(\Z/2\Z)$ for $C$, a generic $n$-pointed rational curve, and $\tree$, a trivalent tree with $n$ leaves.   This follows work of Buczy{\'n}ska and Wi{\'s}niewski,  who proved in \cite{BW07} that the Hilbert functions of these spaces all agree.   The result of Buczy{\'n}ska and Wi{\'s}niewski left open the question whether or not the toric ideals of the varieties $M_{\tree}(\Z/2\Z)$ lie in the same irreducible component of their Hilbert scheme, this question is answered in the positive by Sturmfels and Xu's result. 

In \cite{Buczynska12}, Buczy{\'n}ska goes on to define a statistical model $M_{\Gamma}(\Z/2\Z)$ for every trivalent graph $\Gamma$, and shows that the Hilbert function of this space depends only on the first Betti number and number of leaves of $\Gamma.$   Corresponding structures appear in \cite{Manon09}, where the second author constructs a flat degeneration $V_{\Gamma}(G)$ of the algebra $V_{C, \vec{p}}(G)$ for each graph $\Gamma$ with $n$ leaves and first Betti number $g$ equal to the genus of $C.$   A higher genus version of the theorem of Sturmfels and Xu is a corollary of these results, namely the algebra $V_{\Gamma}(\SL_2(\C))$ is isomorphic to the coordinate ring of $M_{\Gamma}(\Z/2\Z)$.  In this paper we generalize this relationship to $\Z/m\Z$ and $\SL_m(\C).$  The algebra $V_{\Gamma}(G)$ carries an action of $T^{|E(\Gamma)|}$, where $T \subset G$ is a maximal torus, and $E(\Gamma)$ is the edge set of $\Gamma,$ we let $S_{\Gamma}(G)$ denote the semigroup of the characters for this action, and $\pi_{\Gamma}: V_{\Gamma}(G) \to \C[S_{\Gamma}(G)]$ be the projection map onto the corresponding semigroup algebra.  The following is our main theorem.

\begin{theorem}\label{main}
There is an inclusion of semigroup algebras. 

\begin{equation}
\C[M_{\Gamma}(\Z/m\Z)] \subset \C[S_{\Gamma}(\SL_m(\C))]\\
\end{equation}

\end{theorem}

The results of Buczy{\'n}ska and Wi{\'s}niewski on the group $\Z/2\Z$ lead one to wonder whether or not the Hilbert functions of other group-based phylogenetic statistical models are invariant of the graph $\Gamma.$  This question was addressed by the first author in~\cite{Kubjas12}, where it was shown that this property fails for the group $\Z/2\Z \times \Z/2\Z.$ Donten-Bury and Micha{\l}ek then showed in~\cite{DBM12} that this property also fails for the group $\Z/2\Z \times \Z/2\Z \times \Z/2\Z$ and the cyclic groups $\Z/m\Z$ for $m = 3,4,5,7,8,9$. We observe in Theorem~\ref{theorem:dimensions_of_degree_one_components} that the first value of the Hilbert function is always independent of the graph $\Gamma$.

We prove Theorem \ref{main} by carefully studying the conformal blocks of level $1$ for the special linear groups.  The WZW model is closely related to the classical representation theory of $\SL_m(\C)$ (see for example \cite[Corollary 6.2]{U}), this allows us to transfer the combinatorial devices of $\SL_m(\C)$ into the theory of phylogenetic statistical models. We will show that there is a natural correspondence between the degree one elements of the phylogenetic semigroup associated with $\Z/m\Z$ and a set of so-called $\SL_m(\C)$ BZ triangles. The latter objects are tools for counting $\SL_m(\C)$ tensor product multiplicities. As a consequence, we obtain a graphical interpretation of the degree one elements of these phylogenetic semigroups. Moreover, we show that a graded semigroup algebra $\C[\BZ_{\Gamma}^{\gr}(\SL_3(\C))]$ of these BZ triangles surjects onto the $\Z/3\Z$ phylogenetic statistical models. 

\begin{theorem}\label{BZ}
There is a surjective map of semigroup algebras. 

\begin{equation}
\phi_{\Gamma}: \C[\BZ_{\Gamma}^{\gr}(\SL_3(\C))] \to \C[M_{\Gamma}(\Z/3\Z)]\\
\end{equation}

\end{theorem}

This allows a convenient graphical device for studying the phylogenetic statistical model.  We restrict ourselves to the case $m = 3$ here because, the map $\phi_{\Gamma}$ is an isomorphism for $m = 2$ by the results in \cite{SX10} and \cite{Manon09}, and it is not yet known if $V_{C, \vec{p}}(\SL_m(\C))$ has the combinatorial features enjoyed by the $m =2, 3$ cases in general.  In particular a toric degeneration of this algebra has not been found.   As a consequence of the techniques presented in this paper, such a degeneration would also provide convenient graphical devices for studying the coordinate algebras of $M_{\Gamma}(\Z/m\Z).$ 

Graphical interpretations of BZ triangles have been crucial for solving problems in representation theory~\cite{KT99,GP00,KTW04}. We hope that the relation between BZ triangles and group-based models, and the graphical interpretation of group-based semigroups will provide a new approach to tackle open problems in phylogenetic algebraic geometry such as the Sturmfels-Sullivant conjecture about the maximal degrees of minimal generators of group-based ideals~\cite[Conjecture~29]{SS05}.

Sections~\ref{section:conformal_blocks_I} to~\ref{section:phylogenetic_algebraic_geometry} will be introductory chapters defining basic objects in this article. In Section~\ref{section:conformal_blocks_I}, we will define conformal block algebras, and in Section~\ref{section:conformal_blocks_II}, we will study their representation theoretic properties. In Section~\ref{section:phylogenetic_algebraic_geometry}, the introduction to phylogenetic algebraic geometry and to group-based models will be given. In Section~\ref{section:conformal_block_algebras_and_group_based_models}, we will prove Theorem~\ref{main}. In Section~\ref{section:BZ_triangles_and_group_based_models}, we will establish a connection between BZ triangles and group-based models, and prove Theorem~\ref{BZ}.

\begin{table}[htbp]
\begin{center}% used the environment to augment the vertical space
% between the caption and the table
\begin{tabular}{l p{9.7cm}}
\toprule
$\tree$ & a tree\\
$\Gamma$ & a graph\\
$E(\Gamma)$ & the set of edges of $\Gamma$\\
$V(\Gamma)$ & the set of non-leaf vertices of $\Gamma$\\
$M_{\Gamma}(A)$ & the algebraic variety associated with a group-based model with abelian group $A$ and graph $\Gamma$\\
$P_{\tree}(A)$ & the polytope associated with a group-based model with abelian group $A$ and tree $\tree$\\
$L_{\tree}(A)$ & the lattice generated by the vertices of $P_{\tree}(A)$\\
$R_{\Gamma}(A)$ & the graded affine semigroup associated with a group-based model with abelian group $A$ and graph $\Gamma$\\
$L_{\Gamma}^{\gr}(A)$ & the graded lattice associated with a group-based model with abelian group $A$ and graph $\Gamma$\\
$\overline{R}_{\Gamma}(A)$ & the saturation of $R_{\Gamma}(A)$\\
$R^{\pr}_{\Gamma}(A)$ & the projection of $R_{\Gamma}(A)$ that forgets the grading\\
$\SL_m(\C)$ & the special linear group\\
$\sl_m(\C)$ & the Lie algebra of $\SL_m(\C)$\\
$\omega_i$ & a fundamental weight of $\sl_m(\C)$\\
$\Delta$ & a Weyl chamber of $\sl_m(\C)$\\
$\lambda \in \Delta$ & a dominant weight of $\sl_m(\C)$\\
$V(\lambda)$ & the irreducible representation of $\SL_m(\C)$ with highest weight $\lambda$\\
$L \in \Z_{\geq 0}$ & the level\\
$\Delta_L$ & the level $L$ alcove of $\Delta$\\
$\alpha_{ij}$ & a positive root of $\sl_m(\C)$\\
$\mathcal{W}_m$ & the weight lattice of $\sl_m(\C)$\\
$\mathcal{R}_m$ & the root lattice of $\sl_m(\C)$\\
$\overline{\mathcal{M}}_{g, n}$ & the moduli of stable genus $g$, $n$-marked curves\\
$(C, \vec{p}) \in \overline{\mathcal{M}}_{g, n}$ & a $n$-marked genus $g$ stable curve\\
$V_{C, \vec{p}}(\vec{\lambda}, L)$ & the space of $\vec{\lambda}$ conformal blocks on $(C, \vec{p})$ of level $L$\\
$V_{C, \vec{p}}(G)$ & the direct sum of the spaces of conformal blocks on $(C, \vec{p})$\\
$V_{\Gamma}(G)$ & the degeneration of $V_{C, \vec{p}}(G)$ associated to $\Gamma$\\
$V(G)$ & the algebra of conformal blocks\\
$S_{\Gamma}(\SL_m(\C))$ & the affine semigroup associated to $V_{\Gamma}(\SL_m(\C))$\\
$\BZ(\SL_m(\C))$ & the affine semigroup of BZ triangles for $\SL_m(\C)$\\
$\pr(\BZ(\SL_m(\C)))$ & the projection of $\BZ(\SL_m(\C))$ to boundaries\\
$\BZ^{\gr}(\SL_m(\C))$ & the graded affine semigroup of BZ triangles for $\SL_m(\C)$\\
\bottomrule
\end{tabular}
\end{center}
\label{tab:TableOfNotation}
\caption{Table of notation}
\end{table}

\section{Essentials of conformal blocks}\label{section:conformal_blocks_I}

\subsection{Basics of $\SL_m(\C)$ representation theory}

We recall a few basics of the representation theory of the special linear group and its Lie algebra, for what follows see the book of Fulton and Harris, \cite{FH}. Recall that $\SL_m(\C)$ is the set of $m \times m$ invertible matrices with determinant $1$, and its Lie algebra $\sl_m(\C)$ is the set of $m\times m$ matrices with trace $0.$ A representation of these objects is a vector space $V$ along with a map of groups $\SL_m(\C) \to \GL(V),$ and respectively a map of Lie algebras $\sl_m(\C) \to \mathrm{M}_{m\times m}(\C).$ 

Representations form a category, and it is a classical result that the categories of finite dimensional representations of $\SL_m(\C)$ and $\sl_m(\C)$ are the same. Representations can be combined in a number of ways, most important to us are the direct sum $V \oplus W$ and the tensor product $V\otimes W,$ which are given by the expected operations on the underlying vector spaces. 

When a representation has no proper sub-representations, it is said to be irreducible.  It can be shown that any representation in $\mathrm{Rep}(m)$ can be written as a direct sum of irreducible representations in a unique way. For this reason, a classification of the irreducible representations of $\SL_m(\C)$ suffices to describe $\mathrm{Rep}(m).$ These irreducibles are in bijection with decreasing lists of non-negative integers of length $m$ which end in $0$.

\begin{equation}
(\lambda_1, \ldots, \lambda_{m-1},0) = \lambda \to V(\lambda)\\
\end{equation}

Such lists are known as weights, and they are the integer points of simplicial cone $\Delta.$ This set is generated under addition by the weights $\omega_1 = (1, \ldots,0), \omega_2 = (1, 1, \ldots,0), \ldots, \omega_{m-1} = (1, \ldots, 1,0).$  These are known as fundamental weights, and they correspond to the exterior powers of $\C^m$. 

\begin{equation}
V(\omega_k) = \bigwedge^k(\C^m)\\
\end{equation}

Any representation $V(\lambda)$ can be viewed as a $(\C^*)^{m-1}$ representation by restricting
the action of $\SL_m(\C)$ to its maximal torus of diagonal matrices.  Since this is an abelian group, 
$V(\lambda)$ decomposes into $1$-dimensional $(\C^*)^{m-1}$ isotypical spaces, called weight spaces. 
Each such space is labelled by a corresponding $(\C^*)^{m-1}$ character, which can be identified with an $(m-1)$-tuple of integers.  The set of all such tuples forms a lattice $\mathcal{W}_m$ in the linear space $\R^m/(1, \ldots, 1)\R$ called the weight lattice.  This lattice is generated by the fundamental weights $\omega_i.$

The weights of the representation $V(\omega_k)$ are easy to compute, they correspond to the exterior powers $z_{i_1}\wedge \ldots \wedge z_{i_k}$, and are in bijection with $m$-length $0, 1$ vectors with exactly $k$ $1'$s. In particular, the subspace (the "weight space") of $V(\omega_k)$ with a particular weight is always $1$-dimensional. 

The weight lattice contains the vectors $\alpha_{ij} = (0, \ldots, 1, \ldots, -1, \ldots 0), 1 \leq i < j \leq m,$ which are known as the roots of the Lie algebra $\sl_m(\C)$, i.e. these are the characters which appear in the decomposition of $\sl_m(\C)$ as a representation of the group of diagonal matrices in $\SL_m(\C).$ The root lattice $\mathcal{R}_m \subset \mathcal{W}_m$ is the lattice generated by the roots.
It is a simple exercise to show the following. 

\begin{equation}
\Z/ m\Z = \mathcal{W}_m /\mathcal{R}_m\\
\end{equation}

As we will see, this is the principal reason that $\SL_m(\C)$ conformal blocks are related to phylogenetic statistical models for the group $\Z/m\Z.$

\subsection{Graphs and curves}

In this subsection we recall some of the combinatorial notions needed to discuss algebras of conformal blocks and phylogenetic
statistical models.  From now on $\Gamma$ denotes a graph, we say $\Gamma$ is of type $g, n$ if its first Betti number is $g$ and
it has $n$ leaves.  We let $V(\Gamma)$ be the set of non-leaf vertices of $\Gamma.$
For two graphs $\Gamma, \Gamma'$ of the same type, we define a map of graphs $\phi: \Gamma' \to \Gamma$ to be a map on the underlying graphs which preserves leaves, obtained by collapsing a set of non-leaf edges of $\Gamma'.$ 

\begin{figure}[htbp]
\centering
\includegraphics[scale = 0.4]{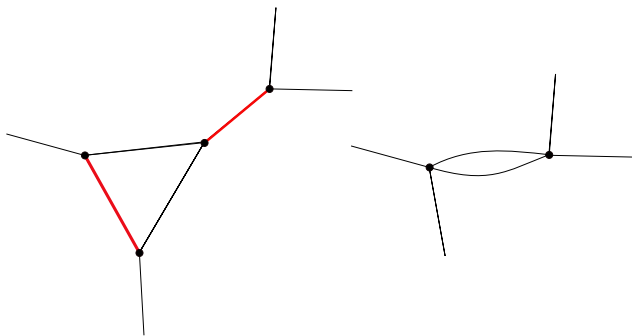}
\caption{The graph on the right is obtained by collapsing the highlighted edges on the left}
\label{fig:collapse}
\end{figure}

For a graph $\Gamma$ we also define a forest $\hat{\Gamma}$, obtained from $\Gamma$ by splitting
each non-leaf edge of $\Gamma.$  The connected components $\Gamma_v$ of $\hat{\Gamma}$ correspond to the
non-leaf vertices $v \in V(\Gamma)$, and each has one distinguished non-leaf vertex.  We let $\pi_{\Gamma}: \hat{\Gamma} \to \Gamma$
be the quotient map which ``glues'' split edges back together.  

Graphs $\Gamma$ act as combinatorial models for points in the moduli of stable curves $\overline{\mathcal{M}}_{g, n}$
of genus $g$ with $n$ marked points.  This stack has a stratification by substacks $\overline{\mathcal{M}}(\Gamma)$, and a curve $C, \vec{p} \in \overline{\mathcal{M}}(\Gamma) \subset \overline{\mathcal{M}}_{g, n}$ is said to be of type $\Gamma.$  Such a curve is generally reducible, with one irreducible component $C_v, \vec{p}_v, \vec{q}_v$ for each non-leaf vertex $v \in V(\Gamma).$ Here $\vec{p}_v$ are the marked points of $C$ which lie on the component $C_v$, and $\vec{q}_v$ are the points on the component $C_v$ which are shared by other components of $C$.  The edges of the labelled tree $\Gamma_v$ are in bijection with the points $\vec{p}_v, \vec{q}_v.$

\begin{figure}[htbp]
\centering
\includegraphics[scale = 0.4]{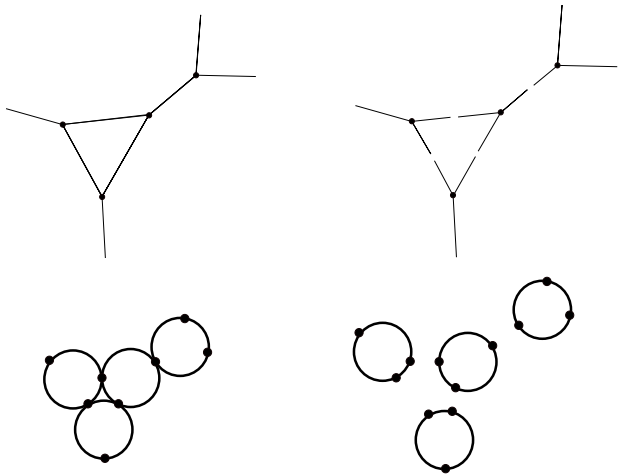}
\caption{The forest $\hat{\Gamma},$ along with the normalization of the associated stable curve}
\label{fig:curvetype}
\end{figure} 

For a reducible curve $C, \vec{p}$ of type $\Gamma$, we can form the normalization $\hat{C}, \vec{p}, \vec{q}$ by splitting
each nodal singularity into a pair of points, creating a disconnected, marked curve. Each pair of points introduced by the normalization
become new marked points in $\hat{C}$.  The combinatorial type of $\hat{C}, \vec{p}, \vec{q}$ is given by the forest $\hat{\Gamma},$
and the connected components $C_i$ of $\hat{C}$ are the irreducible components of $C$, and correspond to the trees $\Gamma_{v_i}.$

The poset of the stratification of $\overline{\mathcal{M}}_{g, n}$ given by the substacks $\overline{\mathcal{M}}(\Gamma)$ is determined
by the graph maps $\pi: \Gamma' \to \Gamma.$  The component $\overline{\mathcal{M}}(\Gamma')$ is in the closure of $\overline{\mathcal{M}}(\Gamma)$
if and only if there is a graph map $\pi: \Gamma' \to \Gamma.$  The lowest strata of $\overline{\mathcal{M}}_{g, n}$ are given by trivalent graphs $\Gamma,$ these are all isolated points in the moduli $\overline{\mathcal{M}}_{g, n}.$

\subsection{Conformal Blocks}

In this subsection we recall some properties of the algebra of conformal blocks, and define the semigroup $S_{\Gamma}(\SL_m(\C)).$ Vector spaces of conformal blocks arise from the Wess-Zumino-Witten (WZW) model of conformal field theory.  We let $\Delta_L$ be the simplex $\{\lambda | \lambda(\alpha_{1,m}) \leq L\}$, where $\alpha_{1, m}$ is the longest root of $\sl_m(\C),$ this polytope is called the level $L$ alcove.  There is a conformal blocks vector space $V_{C, \vec{p}}(\vec{\lambda}, L)$ for each $n$-marked genus $g$ stable curve $C, \vec{p} \in \overline{\mathcal{M}}_{g, n},$ set of $n$ dominant weights $\vec{\lambda} = \{\lambda_1, \ldots, \lambda_n\}$, $p_i \to \lambda_i$, and non-negative integer $L$, known as the level.  Vectors in these spaces are possible partition functions for the WZW model on $C$.  In \cite{TUY}, Tsuchiya, Ueno, and Yamada show that for fixed $\vec{\lambda}, L$, these spaces knit together into a vector bundle over $\overline{\mathcal{M}}_{g, n}.$  This theorem is modified by the second author in \cite{Manon09} with the following theorem.  

\begin{theorem}
There is a flat sheaf $V(\SL_m(\C))$ of algebras on the moduli space $\overline{\mathcal{M}}_{C, \vec{p}}$ of stable curves, such that the fiber over a point $C, \vec{p}$ is isomorphic to the direct sum $V_{C, \vec{p}}(\SL_m(\C)) = \bigoplus_{\vec{\lambda}, L} V_{C, \vec{p}}(\vec{\lambda}, L)$ of $\SL_m(\C)$ conformal blocks on $C, \vec{p}$. 
\end{theorem}

\noindent
This theorem, along with work of Kumar, Narasimhan, Ramanathan \cite{KNR}, Pauly \cite{P}, Beauville, Laszlo, \cite{BL}, and Faltings \cite{F}, establishes that over a smooth curve $C, \vec{p},$ the algebra $V_{C, \vec{p}}(\SL_m(\C))$ is isomorphic to the total coordinate ring of the moduli of rank $m$ vector bundles on $C$ with parabolic structure at the marked points $\vec{p}.$

Conformal blocks come with a number of structural features. For a graph $\Gamma$ and a stable curve $C, \vec{p}$ of type $\Gamma$ we can take the smooth normalization, $\tilde{C}, \vec{p}, \vec{q}_1, \vec{q}_2 $ by pulling apart the double point singularities, resulting in pairs of marked points $\vec{q}_1, \vec{q}_2$. The following theorem expresses how conformal blocks behave with respect to this operation, it is due to Tsuchiya, Ueno, and Yamada, \cite{TUY} and also Faltings \cite{F}.  

\begin{theorem}\label{factor}
For $C, \vec{p}$ a stable curve with normalization $\tilde{C}, \vec{p},\vec{q}_1, \vec{q}_2$, there is an isomorphism of vector spaces, canonical up to choice of scalar for each $\vec{\eta} = (\eta_1, \ldots, \eta_m)$ summand, where $|\vec{q}_i| = m.$ 

\begin{equation}
V_{C, \vec{p}}(\vec{\lambda}, L) = \bigoplus_{\eta_j \in \Delta_L} V_{\tilde{C},\vec{p}, \vec{q}_1, \vec{q}_2 }(\vec{\lambda}, \vec{\eta}, \vec{\eta}^*, L)\\
\end{equation}

Here the sum is over all assignments of weights $\eta_j \in \Delta_L.$ 
\end{theorem} 

This theorem allows us to express properties about conformal blocks over a general curve $C, \vec{p}$
in terms of simpler curves.  In particular, for every graph $\Gamma$ with first Betti number equal to $g$ and $n$ leaves there is a stable curve $C, \vec{p}$ with normalization $\tilde{C}, \vec{p},\vec{q}_1, \vec{q}_2$ equal to a disjoint union of  marked copies of the rational curve $\mathbb{P}^1$.  Spaces of conformal blocks over non-connected curves are tensor products of the spaces of conformal blocks over their connected components, so this implies that we can effectively study the conformal blocks for any curve by understanding these spaces over marked projective spaces.

This approach extends to the algebraic structure introduced in \cite{Manon09}. We fix a marked curve $C, \vec{p}$ of type $\Gamma$ with normalization $\hat{C}, \vec{p}, \vec{q},$ and we consider the following tensor product. 

\begin{equation}
\bigotimes_{v \in V(\Gamma)} V_{C_{v}, \vec{p}_v, \vec{q}_v}(\SL_m(\C))\\
\end{equation}

\noindent
This algebra is a multigraded sum of tensor product spaces of conformal blocks. 

\begin{equation}
\bigotimes_{v \in V(\Gamma)} V_{C_{v}, \vec{p}_v, \vec{q}_v}(\vec{\lambda}_v, \vec{\eta}_v, L_v)\\
\end{equation}

Here $\vec{\lambda}_v$ are the dominant weights assigned to the marked points of $C_v$ which correspond
to marked points on the original curve $C,$ and $\vec{\eta}_v$ are the dominant weights assigned to marked points which are introduced by the normalization procedure.  Recall that there is a way to pair these weights $\eta_{v_i} \leftrightarrow \eta_{v_j}.$

\begin{definition}
The algebra $V_{\Gamma}(\SL_m(\C))$ is the sub-algebra of 
\begin{displaymath}
\bigotimes_{v \in V(\Gamma)} V_{C_{v}, \vec{p}_v, \vec{q}_v}(\SL_m(\C)) 
\end{displaymath}
defined by the components $V_{C_{v}, \vec{p}_v, \vec{q}_v}(\vec{\lambda}_v, \vec{\eta}_v, L_v)$ with $L_{v_i} = L_{v_j}$ for all $v \in V(\Gamma)$, 
and $\eta_{v_i} = \eta_{v_j}^*$ for all paired dominant weights. 
\end{definition}

This algebra is related to the algebra of conformal blocks $V_{C, \vec{p}}(\SL_m(\C))$ for a curve of genus $g$ with $n$ markings by the following theorem from \cite{Manon09}.

\begin{theorem}
There is a flat degeneration of commutative algebras $V_{C, \vec{p}}(\SL_m(\C))$ to $V_{\Gamma}(\SL_m(\C)).$ 
\end{theorem}

\begin{definition}
We define the semigroup $S_{\Gamma}(\SL_m(\C)) \subset \Delta^{|E(\hat{\Gamma})|}\times \Z_{\geq 0}$ to be the set of pairs $(\omega, L)$, where
 $\omega: E(\hat{\Gamma}) \to \Delta^{|E(\hat{\Gamma})|}$ and $L \in \Z_{\geq 0}$, such that the component of $V_{\Gamma}(\SL_m(\C))$
defined by $(\omega, L)$ is non-zero. 
\end{definition}

There is a map of algebras, $V_{\Gamma}(\SL_m(\C)) \to \C[S_{\Gamma}(\SL_m(\C))]$ given by sending conformal blocks to their weight and level information. Note that if $\pi: \Gamma' \to \Gamma$ is a map of graphs, then there is a corresponding map of semigroups $S_{\Gamma'}(\SL_m(\C)) \to S_{\Gamma}(\SL_m(\C))$ defined by forgetting the weights on the edges collapsed by $\pi.$

\section{Conformal blocks and invariants in tensor products}\label{section:conformal_blocks_II}

In this section we discuss the construction of the space of conformal blocks over a genus $0$, triple marked curve.  This involves spaces of invariant vectors in tensor products of $\SL_m(\C)$ representations in a non-trivial way. We use this construction to revisit a well-known combinatorial classification of level $1$ $\SL_m(\C)$ conformal blocks.  The combinatorics of these special conformal blocks allow us to make the connection with phylogenetic statistical models in Section~\ref{section:conformal_block_algebras_and_group_based_models}.

\subsection{Spaces of conformal blocks in spaces of invariants}

In order to build the space $V_{0, 3}(\lambda, \eta, \mu, L)$ of conformal blocks
we study the action of $\SL_m(\C)$ on the tensor product $V(\lambda) \otimes V(\eta) \otimes V(\mu).$
We let $[V(\lambda) \otimes V(\eta) \otimes V(\mu)]^{\SL_m(\C)}$ be the space of vectors fixed by 
this action. The space of conformal blocks can be constructed as a subspace
of this space of invariants. 

\begin{equation}
V_{0, 3}(\lambda, \eta, \mu, L) \subset [V(\lambda) \otimes V(\eta) \otimes V(\mu)]^{\SL_m(\C)}\\
\end{equation}

We introduce the subgroup $\SL_2(1, m) \subset \SL_m(\C)$, this is the copy of $\SL_2(\C)$ inside of $\SL_m(\C)$ which acts on the basis vectors with indices $1, m$ in the standard representation of $\SL_m(\C).$  The group $\SL_2(1, m)$ is composed of the matrices of the following form. 

$$
\left( \begin{array}{cccc}
 a & 0 & \cdots & b \\
 0 & 1 & \cdots & 0\\
 \vdots & \vdots & \ddots  & \vdots\\
 c & 0 &\cdots & d \\ 
\end{array} \right)
$$

The subgroup $\SL_2(1, m)$ corresponds to the longest root of $\SL_m(\C)$ in its standard ordering of roots. 
By way of the inclusion $\SL_2(1, m) \subset \SL_m(\C)$, every $\SL_m(\C)$ representation can be restricted
to a representation of $\SL_2(1, m)$, and decomposed along the irreducible representations of $\SL_2(\C).$

\begin{equation}
V(\lambda) = \bigoplus_i W_{\lambda, i}\otimes V(i)\\
\end{equation}
$$V(\eta) = \bigoplus_j W_{\eta, j}\otimes V(j)$$
$$V(\mu) = \bigoplus_k W_{\mu, k}\otimes V(k)$$

We let $W_L \subset V(\lambda) \otimes V(\eta) \otimes V(\mu)$ be the following subspace,

\begin{equation}
\bigoplus_{i + j + k \leq 2L}  [W_{\lambda, i}\otimes V(i)] \otimes [W_{\eta, j}\otimes V(j)] \otimes [W_{\mu, k}\otimes V(k)].\\
\end{equation}

\noindent
For the following proposition see \cite[Corollary 6.2]{U}. 

\begin{proposition}\label{ueno}

The space of conformal blocks $V_{0, 3}(\lambda, \eta, \mu, L)$ can be identified with the following subspace of $[V(\lambda)\otimes V(\eta) \otimes V(\mu)]^{\SL_m(\C)},$ 

\begin{equation}
V_{0, 3}(\lambda, \eta, \mu, L) = W_L \cap [V(\lambda) \otimes V(\eta) \otimes V(\mu)]^{\SL_m(\C)}.\\
\end{equation}
\end{proposition}

\subsection{Level $1,$ genus $0$, $3$-marked conformal blocks}

A consequence of Proposition \ref{ueno} is that the dimension of $V_{0, 3}(\lambda, \eta, \mu, L)$ is bounded above by the dimension of the space $[V(\lambda) \otimes V(\eta) \otimes V(\mu)]^{\SL_m(\C)}.$  This is particularly useful in the case we consider here, when the level $L$ is restricted to be $1.$  The level $1$ alcove $\Delta_1$ is precisely the $0$ weight and the fundamental weights $\omega_i,$ so all conformal blocks we consider will now have only these weights in their weight datum.  We make use of the following fact from the representation theory of $\SL_m(\C).$

\begin{lemma}
The tensor product invariant space $[V(\lambda) \otimes V(\omega_i) \otimes V(\mu)]^{\SL_m(\C)}$ is
multiplicity free, where $\omega_i$ is the $i$-th fundamental weight of $\SL_m(\C)$. 
\end{lemma}

\begin{proof}
By a general theorem due to Zhelobenko \cite{Zh} (see also \cite{Manon10}), the space  $[V(\lambda) \otimes V(\omega_i) \otimes V(\mu)]^{\SL_m(\C)}$ can be identified with a subspace of a weight space of $V(\omega_i) = \bigwedge^i(\C^m).$  As we remarked above, all weight spaces of these representations are multiplicity free. 
\end{proof}

\begin{proposition}
The space $V(\omega_i)\otimes V(\omega_j) \otimes V(\omega_k)$ has an invariant $\SL_m(\C)$ vector if and only if $i + j + k = 0$ modulo $m$.
\end{proposition}

\begin{proof}
If the tensor product $V(\omega_i) \otimes V(\omega_j) \otimes V(\omega_k)$ has an invariant, then  the weight $\omega_j + \omega_k - \omega_{m-i}$ is a member of the root lattice of $\SL_m(\C).$ This means that the class $[\omega_j] + [\omega_k] - [\omega_{m-i}] = 0$ in the group $\mathcal{W}_m / \mathcal{R}_m.$ As $\mathcal{W}_m / \mathcal{R}_m= \Z/m\Z$ with $[\omega_i] = i$, this implies that there is an invariant vector
in  $V(\omega_i) \otimes V(\omega_j) \otimes V(\omega_k)$ only if $i +j + k = 0$ modulo $m.$ This proves
that the condition $i + j + k \in m\Z$ is necessary.

For sufficiency, note that $i + j + k$ above must be $m$ or $2m$. The space $V(\omega_i) \otimes V(\omega_j) \otimes V(\omega_k)$ has an invariant vector if and only if $V(\omega_{m-i}) \otimes V(\omega_{m-j}) \otimes V(\omega_{m-k})$ does as well by duality, so we assume without loss of generality that $i + j + k = m$. The representation $V(\omega_k)$ is the exterior product $\bigwedge^k(\C^m)$ as a vector space, and has a basis of exterior forms $z_I = z_{i_1} \wedge \ldots \wedge z_{i_k}$, $I = \{i_1, \ldots, i_k\}.$ The determinant form $z_1\wedge \ldots \wedge z_m$ can be expressed in these basis members. 

\begin{equation}\label{det}
z_1\wedge \ldots \wedge z_m = \sum_{I \coprod J \coprod K = [m]} (-1)^{\sigma(I,J,K)} z_I \otimes z_J \otimes z_K\\
\end{equation}

Here $(-1)^{\sigma(I, J, K)}$ is a sign that depends on the indices $I, J, K$. Recall that the determinant form is the unique basis member of $\bigwedge^m(\C^m),$ which is the trivial representation of $\SL_m(\C).$ This implies that $V(\omega_i) \otimes V(\omega_j)\otimes V(\omega_k)$ has an invariant vector when $i + j + k =m,$ and when $i + j + k = 2m$ by duality.

\end{proof}

Now we can determine the dimension of the spaces $V_{0, 3}(\omega_i, \omega_j, \omega_k, 1)$.  Note
that this space has dimension $1$ or $0,$ and the same duality considerations allow us to restrict to the case $i + j + k = m,$ as above.  

\begin{proposition}
The space $V_{0, 3}(\omega_i, \omega_j, \omega_k, 1)$ with $i + j + k = 0$ modulo $m$ has dimension $1.$ 
\end{proposition}

\begin{proof}
We determine the space $W_{\omega_i,s} \otimes V(s) \otimes W_{\omega_j,r} \otimes V(r) \otimes W_{\omega_k,t} \otimes V(t) \subset V(\omega_i) \otimes V(\omega_j) \otimes V(\omega_k)$ which contains the element $z_I \otimes z_J \otimes z_K$ for each element in the expansion of the determinant
in Equation (\ref{det}).  This is determined by which sets $I, J, K$ contain the indices $1$ and $m.$  

If both $1, m$ are contained in the same index set, say $J$, then this form restricts to the determinant form $z_1 \wedge z_m$ of $\SL_2(1, m).$ This implies that $z_I \otimes z_J \otimes z_K$ lies in a component of the form  $W_{\omega_i,0} \otimes V(0) \otimes W_{\omega_j,0} \otimes V(0) \otimes W_{\omega_k,0} \otimes V(0).$  If $1, m$ are in different index sets, say $I, K$, then these give basis members $z_1$ and $z_m$ of the representation $V(1)$ of $\SL_2(1, m)$. This implies that $z_I \otimes z_J \otimes z_K$ lies in a component of the form $W_{\omega_i,1} \otimes V(1) \otimes W_{\omega_j,0} \otimes V(0) \otimes W_{\omega_k,1} \otimes V(1).$  It follows that the maximum value obtained by $s + r + t$ is $2$, and this value is obtained. 
\end{proof}

\subsection{Level $1$ conformal blocks}\label{subsection:level_one_conformal_blocks}

Now we extend the classification of level $1$ conformal blocks to curves $C$
of any genus and number of markings. Our principal tool is the factorization property, Theorem \ref{factor}. This allows us to express a space of level $1$ conformal blocks $V_{C, \vec{p}}(\omega_{i_1}, \ldots, \omega_{i_n}, 1)$ in terms of the spaces of genus $0$, $3$-marked conformal blocks studied above.
There is one such expression for each trivalent graph $\Gamma$.

\begin{equation}
V_{C, \vec{p}}(\omega_{i_1}, \ldots, \omega_{i_n}, 1) = \bigoplus [\bigotimes_{v \in V(\hat{\Gamma})} V_{0, 3}(\omega_{i_v}, \omega_{j_v}, \omega_{k_v}, 1)]\\
\end{equation}

 Each of the tensor product spaces appearing on the right hand side of this equation has dimension $1$ or $0$, so the space $V_{C, \vec{p}}(\omega_{i_1}, \ldots, \omega_{i_n}, 1)$ has a basis in bijection with the set of labellings of the edges of $\hat{\Gamma}$ by fundamental weights such that the following conditions are satisfied. 

\begin{enumerate}\label{con}
\item For each $v \in V(\hat{\Gamma})$ $i_v + j_v + k_v = 0$ modulo $m$.
\item If $\omega_{j_v}$ and $\omega_{k_{v'}}$ are assigned to edges which agree under $\pi: \hat{\Gamma} \to \Gamma$, then $j_v + k_{v'} = 0$ modulo $m.$
\end{enumerate}

The second condition above follows from duality. This allows us to combinatorially determine the dimension of any space of level $1$ $\SL_m(\C)$ conformal blocks. 

\begin{theorem}\label{theorem:dimensions_of_degree_one_components}
The space $V_{C, \vec{p}}(\omega_{i_1}, \ldots, \omega_{i_n}, 1)$ has dimension $0$ or $m^g$.
It has dimension $m^g$ precisely when $\sum i_j = 0$ modulo $m.$
\end{theorem}

\begin{proof}
It is a luxury of the factorization theorem that we may use any trivalent graph $\Gamma$
we wish in the analysis of the dimension of $V_{C, \vec{p}}(\omega_{i_1}, \ldots, \omega_{i_n}, 1)$.
We employ the graph $\Gamma_{g, n}$, depicted in Figure~\ref{fig:virus}.

\begin{figure}[htbp]
\centering
\includegraphics[scale = 0.4]{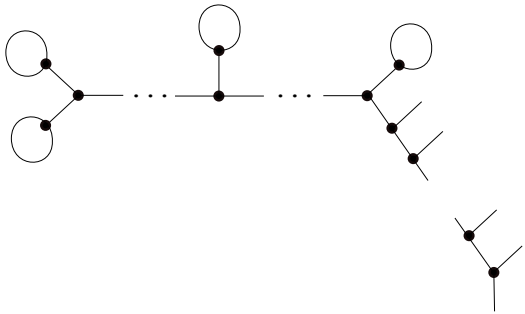}
\caption{The graph $\Gamma_{g, n}$}
\label{fig:virus}
\end{figure}

The paired edges in $\hat{\Gamma}_{g, n}$ which correspond to the loops must each
be assigned dual fundamental weights.  This forces the third edge in each of these tripods
to be assigned $0$.  This forces each edge to the left of $e$ to be assigned a $0$. 
Note that there are $m^g$ such assignments of fundamental weights and their duals on 
these loop edges.  

This reduces the problem to showing that when $g = 0,$ there is a unique
labelling if and only if $\sum i_j = 0$ modulo $m$.  We have already seen that
this is the case if $n = 3.$ The general result follows from the observation that
in the case of a trivalent tree $\tree,$ an assignment of weights to the leaves of $\hat{\tree}$ determines the assignment on the rest of the tree by the equation $i_v + j_v + k_v = 0$ modulo $m$. 
This follows because every trivalent tree $\tree$ must have a vertex $v$ which is connected to two leaves.  Using the equation $i_v + j_v + k_v = 0$ modulo $m$ to force the weighting on the third edge
connected to this tripod reduces the case by $1$ leaf, so we obtain the result by induction. 
\end{proof}

There are $m^{n-1}$ $n$-tuples of fundamental weights $(\omega_{i_1}, \ldots, \omega_{i_n})$
such that $\sum i_j = 0$ modulo $m$. It follows that the space of all $\SL_m(\C)$ conformal blocks of level $1$ has dimension $m^{g + n - 1}.$

\section{Phylogenetic algebraic geometry}\label{section:phylogenetic_algebraic_geometry}

Phylogenetic algebraic geometry studies algebraic varieties associated with statistical models on phylogenetic trees, see \cite{ERSS05} for an introductory article on the topic. These statistical models on phylogenetic trees can be parametrized by polynomial maps that give the explicit description of the associated algebraic variety. 

Group-based models are statistical models invariant under the action of a finite abelian group. Algebraic varieties associated with group-based models are toric after the change of coordinates given by the discrete Fourier transform~\cite{ES,SSE}. By the well-known correspondence between toric varieties and polyhedral fans, many properties of group-based models can be studied by considering corresponding discrete objects. In the following, we will define lattice polytopes and affine semigroups associated with group-based models.

\begin{definition}\label{symmetric_group_based_model}
Let  $A$ be an additive finite abelian group and $\tree$ a tree. Let $\hat{\tree}$ be the forest obtained from $\tree$ by splitting every non-leaf edge of $\tree$. Label the edges of $\hat{\tree}$ with unit vectors indexed by the elements of $A$ such that
\begin{enumerate}
 \item at each inner vertex the indices of unit vectors sum up to zero,
 
 \item the indices of unit vectors corresponding to an inner edge of $\tree$ sum up to zero.
\end{enumerate}
 Define the polytope $P_{\tree}(A)$ as the convex hull of vectors corresponding to all such labellings of the edges of $\tree$. Define the lattice $L_{\tree}(A)$ to be the lattice generated by the vertices of $P_{\tree}(A)$. We denote the polytope associated with a group $A$ and the tripod  by $P_{0,3}(A)$ and the corresponding lattice by $L_{0,3}(A)$.
\end{definition}

Let $e$ be an edge of $\hat{\tree}$ and $a$ an element of $A$. Denote the coordinate corresponding to the tuple $(e,a)$ by $x^{e}_a$.  
For $x\in P_{\tree}(A)$, we have
\begin{equation}
\sum _{a\in A} x^{e}_a=1.
\end{equation}
Hence the projection of $P_{\tree}(A)$ that forgets the coordinates corresponding to $0\in A$ is lattice equivalent to $P_{\tree}(A)$. From now on we will consider this projection of $P_{\tree}(A)$ and we will also denote it by $P_{\tree}(A)$.

\begin{example}\label{example:Jukes_Cantor_binary_model}
The polytope associated with $\Z/2\Z$ and a claw tree $\tree$ is 
\begin{equation}
P_{\tree}(\Z/2\Z)=\conv \{x\in \{0,1\}^{E(\tree)}|\sum _{e\in E(\tree)} x_e \textrm{ is even}\},
\end{equation}
where $E(\tree)$ denotes the set of edges of $\tree$. In particular, the polytope associated with $\Z/2\Z$ and tripod is
\begin{equation}
 P_{0,3}(\Z/2\Z)=\conv \{(0,0,0),(0,1,1),(1,0,1),(0,1,1)\}.
\end{equation}
The group-based model associated with group $\Z/2\Z$ is called the Jukes-Cantor binary model.
\end{example}

Note the similarity of the conditions (1) and (2) in Definition~\ref{symmetric_group_based_model} with the conditions at the beginning of Section~\ref{subsection:level_one_conformal_blocks}. These conditions will be the main ingredient for establishing the relationship between group-based models, conformal block algebras and BZ triangles in Sections~\ref{section:conformal_block_algebras_and_group_based_models} and~\ref{section:BZ_triangles_and_group_based_models}. 

Definition~\ref{symmetric_group_based_model} differs from the standard definition of a lattice polytope associated with a group-based model and a tree. For the standard definition, one needs to direct the edges of $\tree$ and the definition depends on the orientation we choose, see for example~\cite{SS05} and~\cite[Section 3.4]{Sullivant07}. Different ways of directing the edges of $\tree$ give lattice equivalent lattice polytopes. Also the polytope in Definition~\ref{symmetric_group_based_model} is lattice equivalent to all of them. 

The advantage of Definition~\ref{symmetric_group_based_model} is that it does not depend on the orientation of $\tree$ we choose. This makes it simpler to establish a connection with conformal block algebras and BZ triangles. Moreover, Definition~\ref{symmetric_group_based_model} makes it possible to define affine semigroups associated with a group $A$ and a graph $\Gamma$ generalizing the definition of the phylogenetic semigroup for the group $\Z /2\Z$ and a graph $\Gamma$ in~\cite{Buczynska12,BBKM13}.

\begin{definition}\label{def:phylogenetic_semigroup_on_a_tree}
The affine semigroup associated with an additive finite abelian group $A$ and a tree $\tree$ is
\begin{displaymath}
R_{\tree}(A)=\N((P_{\tree}(A)\cap L_{\tree}(A))\times \{1\}).
\end{displaymath}
We denote the semigroup associated with a group $A$ and the tripod by $R_{0,3}(A)$. The associated graded lattice is
\begin{displaymath}
L_{\tree}^{\gr}(A)=L_{\tree}(A) \oplus \Z.
\end{displaymath}
 
\end{definition}

 Using the definition for trees, we can define the semigroup associated with a group-based model and a graph.
Let $\Gamma$ be a graph with first Betti number $g$. We associate a tree $\tree$ with $g$ pairs of distinguished leaves with $\Gamma$. Choose an edge $e_1$ of $\Gamma$ such that splitting $e_1$ gives a graph with first Betti number $g-1$. Denote the new leaf edges obtained from $e_1$ by $\overline{e_1}$ and $\underline{e_1}$. Repeating this operation, we get an associated tree $\tree$ with $g$ pairs of distinguished leaves $\{\overline{e_1},\underline{e_1}\},\{\overline{e_2},\underline{e_2}\},\ldots ,\{\overline{e_g},\underline{e_g}\} $. 

\begin{definition}\label{def:phylogenetic_semigroup_on_a_graph}
The affine semigroup associated with an additive finite abelian group $A$ and a graph $\Gamma$ is
\begin{displaymath}
R_{\Gamma}(A)=R_{\tree}(A) \cap \bigcap_{i\in\{1,\ldots ,g\} \textrm{ and } a\in A} \ker (x_a^{\overline{e_i}}-x_{-a}^{\underline{e_i}}).
\end{displaymath}
The associated graded lattice is
\begin{displaymath}
L_{\Gamma}^{\gr}(A)=L_{\tree}^{\gr}(A) \cap \bigcap_{i\in\{1,\ldots ,g\} \textrm{ and } a\in A} \ker (x_a^{\overline{e_i}}-x_{-a}^{\underline{e_i}}).
\end{displaymath}
\end{definition}

Although the tree $\tree$ associated with $\Gamma$ is not unique, the definitions of $R_{\Gamma}(A)$ and $L_{\Gamma}^{\gr}(A)$ do not depend on the choices we make.

\begin{lemma}\label{lemma:group_based_semigroup}
Let $A$ be an additive finite abelian group and $\Gamma$ a graph. Denote the edges obtained by splitting an inner edge $e$ of $\Gamma$ by $\overline{e}$ and $\underline{e}$. Then
\begin{equation}
R_{\Gamma}(A)= \prod_{v\in V(\Gamma)} R_{0,3}(A) \cap \bigcap _{e \textrm{ an inner edge and } a\in A} \ker(x_a^{\overline{e}}-x_{-a}^{\underline{e}}),
\end{equation}
where $V(\Gamma)$ is the set of non-leaf vertices of $\Gamma$.
\end{lemma}

\begin{proof}
 The statement follows from Definitions~\ref{symmetric_group_based_model},~\ref{def:phylogenetic_semigroup_on_a_tree} and~\ref{def:phylogenetic_semigroup_on_a_graph}. 
\end{proof}

The algebraic variety $M_{\Gamma}(A)$ associated with an abelian group $A$ and a graph $\Gamma$ is defined as the projective spectrum $\Proj(\C[\cone(R_{\Gamma}(A))\cap L_{\Gamma}^{\gr}(A)])$. We will show that the affine semigroup $\cone(R_{\Gamma}(A))\cap L_{\Gamma}^{\gr}(A)$ is the saturation of the affine semigroup $R_{\Gamma}(A)$. Recall that an affine semigroup $S \subset L$ is said to be saturated if $m\cdot u \in S$ if and only if $u \in S$ for all $m \in \Z_{\geq 0}.$ An affine semigroup can always be completed to a saturated affine semigroup $\overline{S}$ by formally including all $u \in L$ such that $m\cdot u \in S.$

\begin{proposition}\label{prop:phylogenetic_semigroup_is_saturation_of_our_semigroup}
Let $\Gamma$ be a graph and $A$ an additive abelian group. Then
\begin{equation}
\overline{R}_{\Gamma}(A) = \cone(R_{\Gamma}(A))\cap L_{\Gamma}^{\gr}(A).
\end{equation}
\end{proposition}
\begin{proof}
We have $\overline{R}_{\Gamma}(A) \subseteq \cone(R_{\Gamma}(A))\cap L_{\Gamma}^{\gr}(A)$, since  ${R}_{\Gamma}(A) \subseteq \cone(R_{\Gamma}(A))\cap L_{\Gamma}^{\gr}(A)$ and $\cone(R_{\Gamma}(A))\cap L_{\Gamma}^{\gr}(A)$ is saturated.

Take a degree $k$ element $u\in \cone(R_{\Gamma}(A))\cap L_{\Gamma}^{\gr}(A)$. Then $\frac{1}{k}u$ is a rational convex combination of degree $1$ elements of $R_{\Gamma}(A)$. Let $m$ be the lowest common denominator of the coefficients in this convex combination. Then $(km)\cdot u\in R_{\Gamma}(A)\subseteq \overline{R}_{\Gamma}(A)$. Since  $\overline{R}_{\Gamma}(A)$ is saturated, we have $u\in \overline{R}_{\Gamma}(A)$. This proves $\overline{R}_{\Gamma}(A) \supseteq \cone(R_{\Gamma}(A))\cap L_{\Gamma}^{\gr}(A)$.
\end{proof}

\begin{corollary}\label{cor:phylogenetic_model_equals_the_proj_of_the_saturation_of_our_semigroup}
Let $\Gamma$ be a graph and $A$ an additive abelian group. Then
\begin{equation}
M_{\Gamma}(A)=\Proj(\C[\overline{R}_{\Gamma}(A)]).
\end{equation}
\end{corollary}

Connections between group-based models, conformal block algebras and BZ triangles have developed from a result about the Jukes-Cantor binary model by Buczy{\'n}ska and Wi{\'s}niewski,  and its generalization by Buczy{\'n}ska. 
\begin{theorem}[\cite{BW07} Theorem~2.24, \cite{Buczynska12} Theorem~3.5]\label{thm:Buczynska}
The Hilbert polynomial of the affine semigroup associated with the Jukes-Cantor binary model on a trivalent graph depends only on the number of leaves and the first Betti number of the graph.
\end{theorem}

It has been shown by the first author~\cite{Kubjas12}, Donten-Bury and Micha{\l}ek~\cite{DBM12} that this theorem cannot be generalized to group-based models with underlying groups $\Z_2\times \Z_2, \Z_2\times \Z_2 \times \Z_2,\Z_3, \Z_4, \Z_5, \Z_7, \Z_8,$ or $\Z_9$. This suggests that the invariance of the Hilbert polynomial of the Jukes-Cantor binary model is a consequence of the flat degenerations of $\SL_2(\C)$ conformal block algebras constructed by Sturmfels and Xu~\cite{SX10}, and the second author~\cite{Manon09}, and is not something typical to group-based models. This hypothesis was further confirmed by the second author in~\cite{Manon12}, where he constructs flat degenerations of $\SL_3(\C)$ conformal block algebras to semigroup algebras associated with $\SL_3(\C)$ BZ triangles on a graph. The Hilbert polynomials of these semigroups depend only on the combinatorial data of the graph, hence $\SL_3(\C)$ BZ triangles on a trivalent graph generalize Theorem~\ref{thm:Buczynska}. These results motivated our investigation of relations between group-based models, BZ triangles and conformal block algebras. In particular, we will show in the next sections that the affine semigroup associated with $\SL_3(\C)$ BZ triangles is closely related to the phylogenetic semigroup associated with $\Z/3\Z$.

\section{Conformal block algebras and group-based models}\label{section:conformal_block_algebras_and_group_based_models}

In this section we start by defining the map from the phylogenetic semigroup $R_{\Gamma}(\Z/m\Z)$ into the semigroup
$S_{\Gamma}(\SL_m(\C))$.

\begin{theorem}\label{theorem:conformal_blocks_group_based_models_inclusion}
Let $\Gamma$ be a graph. Then
\begin{equation}
R_{\Gamma}(\Z/m\Z) \subset S_{\Gamma}(\SL_m(\C)).\\
\end{equation}
\end{theorem}

\begin{proof}
For a tripod, $R_{0,3}(\Z/m\Z) \subset S_{0,3}(\SL_m(\C)) \subset \Delta^{3}\times \Z_{\geq 0}$ is established by results in Section~\ref{subsection:level_one_conformal_blocks}. Here $\Delta$ is a Weyl chamber, considered as a simplicial cone, with extremal rays generated by the unit vectors $e_1, \ldots, e_{m-1}$.  The semigroup $R_{0,3}(\Z/m\Z)$ is the sub-semigroup of this cone generated by elements in $\Delta^{3}\times \{1\}$
where each edge is assigned the generator of such an extremal ray or $e_0=(0,\ldots ,0)$, such that the sum of indices $i$ from the assignment of these generators is $0$ modulo $m$.  Each of these generators corresponds to a unique element of $V_{\Gamma}(\SL_m(\C))$.  The unit vector 
$e_i$ is replaced by the fundamental weight $\omega_i$.  This fundamental weight has image $[i]$ in the quotient group $\Z/m\Z = \mathcal{W}_m/ \mathcal{R}_m.$ 
Similarly, the dual weight $\omega_i^* = \omega_{m-i}$ has image $[m-i]$ in $\Z/m\Z$.  

For a general graph $\Gamma$, we form the covering forest $\hat{\Gamma},$ and note that the following inclusions are immediate. 

\begin{equation}
R_{\hat{\Gamma}}(\Z/m\Z) \subset S_{\hat{\Gamma}}(SL_m(\C)) \subset \Delta^{3|V(\Gamma)|}\times \Z_{\geq 0}\\
\end{equation}

\noindent
For each pair of edges $e, e' \in \hat{\Gamma}$ which are identified by its map to $\Gamma$, there is a subcone $H_{e, e'} \subset \Delta^{3|V(\Gamma)|}\times \Z_{\geq 0}$ of points with $e$ weight dual to $e'$ weight. The theorem now follows from the fact that $R_{\Gamma}(\Z/m\Z) \subset S_{\Gamma}(\SL_m(\C)) \subset \Delta^{|E(\Gamma)|}\times \Z_{\geq 0}$ are obtained from $R_{\hat{\Gamma}}(\Z/m\Z) \subset S_{\hat{\Gamma}}(SL_m(\C)) \subset \Delta^{3|V(\Gamma)|}\times \Z_{\geq 0}$ by intersecting everything in sight with the subcones $H_{e, e'}.$

\end{proof}

\begin{corollary}\label{corollary:saturated_phylogenetic_semigroup_is_included_in_conformal_blocks_semigroup}
 Let $\Gamma$ be a graph. Then
\begin{equation}
\overline{R}_{\Gamma}(\Z/m\Z) \subset S_{\Gamma}(\SL_m(\C)).\\
\end{equation}
\end{corollary}

\begin{proof}
A deep result of Belkale \cite[Section~3.1]{Bel} implies that the semigroup $S_{\Gamma}(\SL_m(\C))$ is saturated, so we also get an inclusion of the saturation $\overline{R}_{\Gamma}(\Z/m\Z)$ in $S_{\Gamma}(\SL_m(\C)).$  
\end{proof}

\begin{proof}[Proof of Theorem~\ref{main}]
Corollary~\ref{corollary:saturated_phylogenetic_semigroup_is_included_in_conformal_blocks_semigroup} gives a corresponding inclusion on semigroup algebras,
\begin{equation}\label{equation:algebra_inclusion}
\C[\overline{R}_{\Gamma}(\Z/m\Z)] \subset \C[S_{\Gamma}(\SL_m(\C))].  
\end{equation} 
Theorem~\ref{main} follows from Equation (\ref{equation:algebra_inclusion}) and Corollary~\ref{cor:phylogenetic_model_equals_the_proj_of_the_saturation_of_our_semigroup}.
\end{proof}

\subsection{$V_{\Gamma}(\SL_2(\C))$}

In the case $m = 2,$ the spaces of conformal blocks $V_{0, 3}(\lambda, \mu, \nu, L)$ for a genus $0$, triple marked curve
are known to be multiplicity free. This follows from Proposition \ref{ueno} above, and the fact that invariant spaces 
for triple tensor products of $\SL_2(\C)$ representations are also multiplicity free.  This implies that the algebra $V_{0, 3}(\SL_2(\C))$
is an affine semigroup algebra on the semigroup given by the weight data $\lambda, \mu, \nu, L$ for which the space  $V_{0, 3}(\lambda, \mu, \nu, L)$ is non-zero.  
This semigroup can be shown to be generated by the vectors $(0, 0, 0, 1),$ $(1,1, 0, 1),$ $(1, 0, 1, 1),$
and $(0, 1, 1, 1) \in \R^4,$ where the last coordinate is the level coordinate $L$. The multiplicity-free property also implies that we have an isomorphism
of commutative algebras, 

\begin{equation}\label{eq}
V_{\Gamma}(\SL_2(\C)) \cong \C[S_{\Gamma}(\SL_2(\C))].\\
\end{equation}

\noindent Furthermore, by Example~\ref{example:Jukes_Cantor_binary_model} we know that $R_{0,3}(\Z/2\Z)$ is generated by $(0, 0, 0, 1),$ $(1,1, 0, 1),$ $(1, 0, 1, 1),$
and $(0, 1, 1, 1) \in \R^4$, thus

\begin{equation}
V_{\Gamma}(\SL_2(\C))\cong \C[R_{\Gamma}(\Z/2\Z)].\\
\end{equation}

The semigroup algebra $\C[R_{\Gamma}(\Z/2\Z)]$ is analyzed by Buczy{\'n}ska, Buczy{\'n}ski, the first author, and Micha{\l}ek in \cite{BBKM13},
where they show it is always generated by elements with $L \leq g(\Gamma)+1$.  In \cite{Manon09}, a particular graph $\Gamma_{g, n}$ is 
found for each $g, n$ such that $\C[S_{\Gamma_{g, n}}(\SL_2(\C))]$ is generated by those elements with $L =1, 2.$  In \cite{Manon12} it is also
shown that this degree of generation is strict when $g, n > 1$ or $4 \leq g.$

\subsection{$V_{\Gamma}(\SL_3(\C))$}

The case $m = 3$ is slightly more complicated than the $m = 2$ case in that Equation (\ref{eq}) no longer holds. Unlike in the $m = 2$ case, the algebra of conformal blocks $V_{0, 3}(\SL_3(\C))$ is not a semigroup algebra.  It has the following presentation by its $L=1$ component, see \cite{Manon12},

\begin{align}
&V_{0, 3}(\SL_3(\C)) = \C[X, Y, Z, P_{12}, P_{23}, P_{31}, P_{21}, P_{32}, P_{13}]/I,
\end{align}
\noindent
where 
\begin{equation}
I=< XYZ - P_{12}P_{23}P_{31} + P_{21}P_{32}P_{13}>.
\end{equation}
Note that three binomial relations can be obtained as degenerations of this algebra, $XYZ - P_{12}P_{23}P_{31},$ $XYZ + P_{21}P_{32}P_{13}$, and $P_{12}P_{23}P_{31} - P_{21}P_{32}P_{13}.$ The ideal generated by the union of these three relations cuts out the phylogenetic statistical model $\C[R_{0,3}(\Z/3\Z)].$

The Corollary~\ref{cor:connection_on_trees} implies that there is a surjection 

\begin{equation}
V_{\tree}(\SL_3(\C)) \to \C[S_{\tree}(\SL_3(\C))] \cong \C[R_{\tree}(\Z/3\Z)].\\
\end{equation}

\noindent
Once again, it can be shown that when $\tree$ is trivalent, the toric ideal which cuts out $S_{\tree}(\SL_3(\C))\cong R_{\tree}(\Z/3\Z)$ is generated by a union of the $3^{|V(\tree)|}$ binomial ideals obtained from the degenerations of $V_{\tree}(\SL_3(\C))$ constructed in \cite{Manon12}. We do not know at this time if a similar statement holds for $\tree$ non-trivalent.

Little is known about the semigroups $R_{\Gamma}(\Z/3\Z)$, $S_{\Gamma}(\SL_3(\C))$  or the algebra $V_{\Gamma}(\SL_3(\C))$ outside of the genus $0$ case.  For a particular graph $\Gamma_{1, n}$ it is shown that $V_{\Gamma_{1, n}}(\SL_3(\C))$ is generated by conformal blocks of levels $1, 2, 3$ in \cite{Manon12}. It would be interesting to obtain a general function which bounds the number of necessary generators
for all $\Gamma$ for both, the algebra $V_{\Gamma}(\SL_3(\C))$ and the semigroup $S_{\Gamma}(\SL_3(\C))$.

\section{Berenstein-Zelevinsky triangles and group-based models}\label{section:BZ_triangles_and_group_based_models}

BZ triangles  provide a combinatorial tool for determining dimensions of triple tensor product invariants. There are several graphical interpretations of BZ triangles such as honeycomb diagrams of Knutson and Tao~\cite{KT99}, web diagrams of Gleizer and Postnikov~\cite{GP00}, and an alternate version of honeycombs by the second author and Zhou~\cite{MZ12}. In this section, we will show that there is a natural injection from the degree one elements of phylogenetic semigroups associated with $\Z/m\Z$ to $\SL_m(\C)$ BZ triangles. As a consequence, we obtain a graphical interpretation of the degree one elements of these phylogenetic semigroups.

It would be interesting to find a graphical interpretation of $\overline{R}_{0,3}(\Z/m\Z)$. If $\overline{R}_{0,3}(\Z/m\Z)$ is not normal, then we need a graphical description of higher degree minimal generators of $\overline{R}_{0,3}(\Z/m\Z)$. Even if $\overline{R}_{0,3}(\Z/m\Z)$ is normal, then higher degree elements of $\overline{R}_{0,3}(\Z/m\Z)$ might have several different honeycombs associated with it. This is caused by the fact that there are less algebraic relations between BZ triangles than between elements of phylogenetic semigroups. We hope that the connection between phylogenetic semigroups and BZ triangles will motivate further research on graphical interpretations of phylogenetic semigroups to approach open problems like the Sturmfels-Sullivant conjecture~\cite[Conjecture~29]{SS05}.

\subsection{BZ Triangles}\label{section:BZ_triangles}

The Littlewood-Richardson rule determines the dimension of the space of triple tensor product invariants $c_{\lambda \mu \nu}=\dim(V_{\lambda}\otimes V_{\mu}\otimes V_{\nu})^{\SL_{m}(\mathbb{C})}$. In~\cite{BZ92},
Berenstein and Zelevinsky  showed  that  $c_{\lambda \mu \nu}$ equals the number of integer points in the intersection of a polyhedral cone with an affine subspace. These integer points are called BZ triangles. The second author and Zhou extended this definition to BZ triangles on trivalent graphs~\cite{MZ12}. We will recall the definition of BZ triangles and projections to corresponding highest weights as in~\cite{BZ92}, and then define BZ triangles on graphs similarly to~\cite{MZ12}.

Let $T_m$ be the set of integer points in the triangle with vertices $(2m-3,0,0)$, $(0,2m-3,0)$ and $(0,0,2m-3)$.
Let $H_m$ be the subset of $T_m$ with all coordinates odd and $G_m$ be the subset of $T_m$ with exactly one coordinate odd.
Geometrically, the elements of $G_m$ correspond to the vertices of a graph consisting of hexagons and triangles, and the elements of $H_m$ correspond to the centers of the hexagons, see Figure~\ref{figure:BZ_triangles}.
\begin{figure}[ht]
\centering
\epsfxsize=360pt\epsfbox{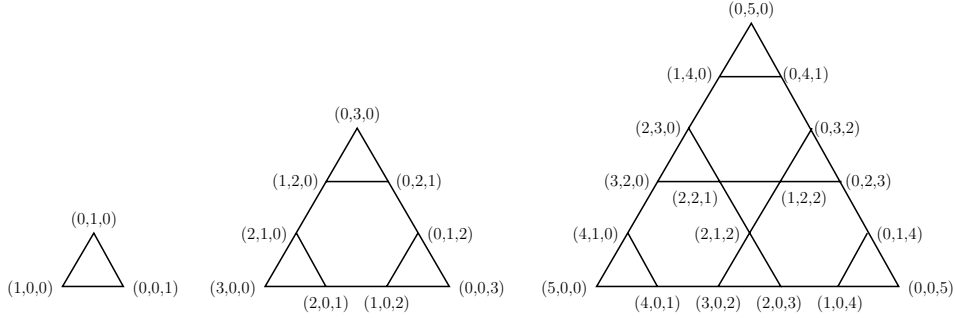} 
\caption{Graphs consisting of hexagons and triangles for $m=2,3,4$}
\label{figure:BZ_triangles}
\end{figure}

Let $L_m$ be the subspace of $ \mathbb{R}^{G_m}$  such that  $x(\xi_1)+x(\xi_2)=x(\xi'_1)+x(\xi'_2)$ for all opposite pairs of edges $[\xi_1,\xi_2]$ and $[\xi'_1,\xi'_2]$ of all hexagons. Define a polyhedral cone $K_m=L_m\cap\mathbb{R}^{G_m}_{\geq 0}$ and an affine semigroup $\BZ(\SL_m(\C))=K_m\cap\Z^{G_m}$. 

\begin{definition}
Elements of the affine semigroup $\BZ(\SL_m(\C))$ are called \textit{BZ triangles}. 
\end{definition}

\begin{example}
Minimal generating sets for $\BZ(\SL_2(\C))$ and $\BZ(\SL_3(\C))$ are listed in Figures~\ref{figure:BZ_rank_one} and~\ref{figure:BZ_rank_two}.
\begin{figure}[ht]
\centering
\epsfxsize=240pt\epsfbox{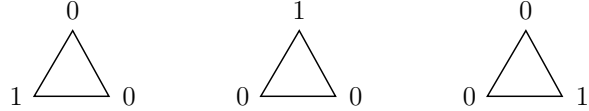}
\caption{Minimal generating set for $\BZ(\SL_2(\C))$}
\label{figure:BZ_rank_one} 
\end{figure}
\begin{figure}[ht]
\centering
\epsfxsize=240pt\epsfbox{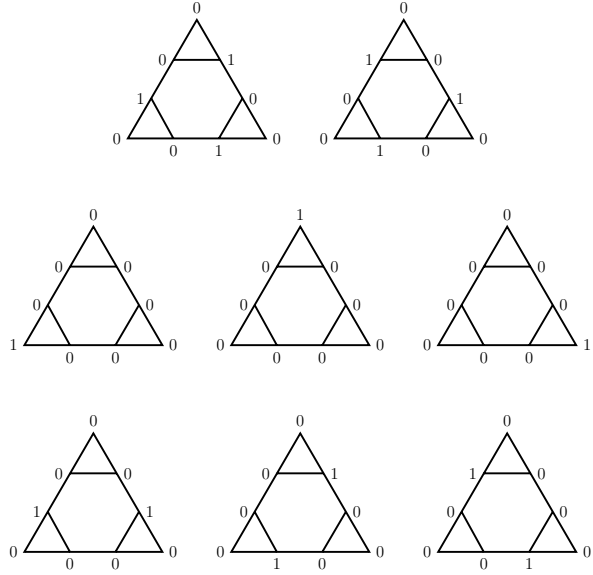}
\caption{Minimal generating set for $\BZ(\SL_3(\C))$}
\label{figure:BZ_rank_two} 
\end{figure}
\end{example}

Honeycombs provide a graphical interpretation of BZ triangles. We use the version of honeycombs defined by the second author and Zhou~\cite[Section 2]{MZ12}. To construct the honeycomb associated with a BZ triangle, we replace each label $x(\xi)$ at a vertex $\xi$ of a hexagon by an edge weighted by $x(\xi)$ from the vertex of the hexagon to the center of the hexagon. A label $x(\xi)$ at the vertex $\xi$ of the underlying triangle is replaced by a vertex weighted by $x(\xi)$. An example of a BZ triangle and the corresponding honeycomb is depicted in Figure~\ref{figure:honeycombs}.

\begin{figure}[ht]
\centering
\epsfxsize=160pt\epsfbox{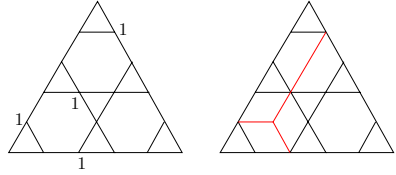}
\caption{A BZ triangle and the corresponding honeycomb}
\label{figure:honeycombs} 
\end{figure}

We identify the lattice of $\SL_m(\C)$-weights with the standard lattice $\Z^{m-1}$ using as the standard basis the fundamental weights $\omega _1=(1,0,0,\ldots ,0)$, $\omega_2=(1,1,0,\ldots ,0),$ $\ldots ,$ $\omega _{m-1}=(1,1,1,\ldots ,1)$. Define a linear projection pr:$L_m\to \mathbb{R}^{3(m-1)}$ to the vector space of triples of $\SL_m(\C)$-weights by 
\begin{equation}
\pr(x)=(\lambda_1,\ldots,\lambda_{m-1};\mu_1,\ldots,\mu_{m-1};\nu_1,\ldots,\nu_{m-1}),
\end{equation}
where
\begin{align*}
\lambda_i &=x(2(m-i)-1,2(i-1),0)+x(2(m-i)-2,2i-1,0),\\
\mu_i &=x(0,2(m-i)-1,2(i-1))+x(0,2(m-i)-2,2i-1),\\
\nu_i &=x(2(i-1),0,2(m-i)-1)+x(2i-1,0,2(m-i)-2)
\end{align*}
for $i\in{1,\ldots,m-1}$. The coordinates of pr$(x)$ are pairwise sums of neighboring labels on the boundary of the underlying triangle starting from the lower left corner and going clockwise around the triangle. In particular, the coordinates $\lambda_i$ correspond to the sums of labels on the northwest edge, $\mu_i$ on the northeast edge, and $\nu_i$ on the south edge of the underlying triangle. 

\begin{example}\label{example:minimal_generators_of_BZ}
The projections of minimal generators of $\BZ(\SL_2(\C))$ as in Figure~\ref{figure:BZ_rank_one} are
\begin{equation}
(1;0;1),(1;1;0),(0;1;1). 
\end{equation}
The projections of minimal generators of $\BZ(\SL_3(\C))$ as in Figure~\ref{figure:BZ_rank_two} are
\begin{align*}
&(1,0;1,0;1,0), (0,1;0,1;0,1),\\
&(1,0;0,0;0,1), (0,1;1,0;0,0), (0,0;0,1;1,0),\\
&(1,0;0,1;0,0), (0,0;1,0;0,1), (0,1;0,0;1,0). 
\end{align*}
\end{example}

\begin{theorem}[\cite{BZ92}, Theorem~1]\label{thm:BZ}
Let $\lambda=\sum \lambda_i\omega _i$, $\mu=\sum \mu_i \omega _i$ and $\nu =\sum \nu_i \omega _i$ be the three highest $\SL_m(\C)$-weights. Then the triple multiplicity $c_{\lambda \mu \nu}$ equals $|\BZ(\SL_m(\C))\cap \pr^{-1}(\lambda,\mu,\nu)|$.
\end{theorem}

Define a linear projection pr$_e:L_m\to \mathbb{R}^{m-1}$ to the vector space of $\SL_m(\C)$-weights for each edge $e$ of the underlying triangle by
\begin{equation}
\pr_e(x)=\left\{ 
\begin{array}{ll}
(\lambda_1,\ldots ,\lambda_{m-1}) & \textrm{if } e \textrm{ is the northwest edge of the triangle},\\
(\mu_1,\ldots ,\mu_{m-1}) & \textrm{if } e \textrm{ is the northeast edge of the triangle},\\
(\nu_1,\ldots ,\nu_{m-1}) & \textrm{if } e \textrm{ is the south edge of the triangle}.
\end{array}\right. 
\end{equation}
The projection pr$_e$ maps a BZ triangle to the corresponding $\lambda, \mu$ or $\nu$. Let $\pr_e^*: L_m\to \mathbb{R}^{m-1}$  be the projection that maps a BZ triangle to the dual $\lambda^*, \mu^*$ or $\nu^*$ of the corresponding weight vector $\lambda, \mu$ or $\nu$, i.e. $\pr_e^*$ is obtained from $\pr_e$ by reversing the order of coordinates.

\begin{figure}[ht]
\centering
\epsfxsize=100pt\epsfbox{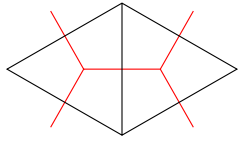}
\caption{The trivalent $4$-leaf tree and the triangle complex dual to it}
\label{figure:triangle_complex}
\end{figure}

Given any trivalent graph $\Gamma$ consider the triangle complex $C$ dual to $\Gamma$, as illustrated in Figure~\ref{figure:triangle_complex}. Let $v$ be an inner vertex of $\Gamma$ and $\Delta_v$ the triangle dual to $v$. Assign a copy of $\SL_m(\C)$ BZ triangles to $v$, denoted by $\BZ(\SL_m(\C))^{(v)}$, such that its underlying triangle is $\Delta_v$. For every edge $e$ adjacent to $v$ define $\pr_e:\BZ(\SL_m(\C))^{(v)}\to \Z^{m-1}$ to be equal to $\pr_{e^*}:\BZ(\SL_m(\C))^{(v)}\to \Z^m$ where $e^*\in C$ is the edge dual to $e$. All projections are taken in the clockwise direction of the corresponding triangle.

\begin{definition}
Let $\Gamma$ be a trivalent graph. The affine semigroup of $\SL_m(\C)$ BZ triangles on $\Gamma$ is
\begin{displaymath}
\BZ_{\Gamma}(\SL_m(\C))=\prod_{v\in V(\Gamma)}\BZ(\SL_m(\C))^{(v)}\cap\bigcap _{e=(v_1,v_2)} \ker (\pr_e(x^{(v_1)})-\pr_e^*(x^{(v_2)})).
\end{displaymath}
If $\Gamma$ is the tripod, then $\BZ_{\Gamma}(\SL_m(\C))=\BZ(\SL_m(\C))$. 
\end{definition}

\begin{example}
Let $\tree$ be the trivalent 4-leaf tree and $C$ the triangle complex dual to $\tree$ as in Figure~\ref{figure:triangle_complex}. A $\SL_3(\C)$ BZ triangle on $\tree$ is shown in Figure~\ref{figure:BZ_triangle_on_tree}. 
\begin{figure}[ht]
\centering
\epsfxsize=100pt\epsfbox{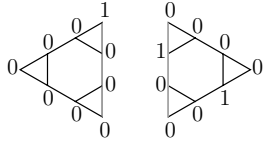}
\caption{BZ triangle on the trivalent 4-leaf tree}
\label{figure:BZ_triangle_on_tree}
\end{figure} 
\end{example}

Similarly to the tripod case, one can construct a honeycomb to any BZ triangle on a trivalent graph, see~\cite[Figure 4]{MZ12} for an example of this construction.

\subsection{Connection Between Group-Based Models and BZ Triangles}\label{section:connection_group-based_models_BZ_triangles}

\begin{lemma} 
Denote the edges of the tripod by $e_1,e_2,e_3$.
The polytope $P_{0,3}(\Z/m\Z)$ has $m^2$ vertices. They are
\begin{itemize}
\item the zero vertex,

\item $x^{e_1}_{[i]}=x^{e_2}_{[m-i]}=1$ for $i\in\{1,\ldots, m-1\}$ and all other coordinates zero,

\item $x^{e_2}_{[i]}=x^{e_3}_{[m-i]}=1$ for $i\in\{1,\ldots, m-1\}$ and all other coordinates zero,

\item $x^{e_1}_{[i]}=x^{e_3}_{[m-i]}=1$ for $i\in\{1,\ldots, m-1\}$ and all other coordinates zero,

\item $x^{e_1}_{[i]}=x^{e_2}_{[j]}=x^{e_3}_{[k]}=1$ for $i,j,k\in\{1,\ldots ,m-1\}$ with $i+j+k = 0$ modulo $m$ and all other coordinates zero.
\end{itemize}
\end{lemma}

\begin{proof}
We can freely choose a label $[i]$ on the edge $e_1$ and a label $[j]$ on the edge $e_2$. The label  $[k]$ on the edge $e_3$ is determined by the first two. The additive group $\Z/m\Z$ has $m$ elements, hence there are exactly $m^2$ possibilities to label the edges of the tripod. All such possibilities are listed above.
\end{proof}

\begin{example}\label{example:correspondence_for_r_equal_to_one_and_two}
The vertices of $P_{0,3}(\Z/2\Z)$ and $P_{0,3}(\Z/3\Z)$ are depicted in Figures~\ref{figure:Z2_vertices} and~\ref{figure:Z3_vertices}. Without the zero vertex, they are in one-to-one correspondence with the minimal generators of the semigroup of $\SL_2(\C)$ and $\SL_3(\C)$ BZ triangles as in Example~\ref{example:minimal_generators_of_BZ}. We will prove in Theorem~\ref{theorem:BZ_triangles_group_based_models_inclusion} that there is a similar connection between  $\SL_m(\C)$ BZ triangles and the group-based model with the underlying group $\Z/m\Z$.
\begin{figure}[ht]
\centering
\epsfxsize=240pt\epsfbox{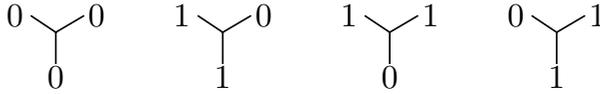}
\caption{Labellings of the tripod with elements of $\Z/2\Z$}
\label{figure:Z2_vertices}
\end{figure}

\begin{figure}[ht]
\centering
\epsfxsize=180pt\epsfbox{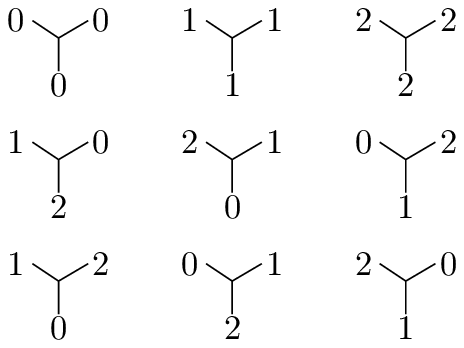}
\caption{Labellings of the tripod with elements of $\Z/3\Z$}
\label{figure:Z3_vertices}
\end{figure}
\end{example}

In the rest of this section we consider the semigroup associated with a group-based model without the grading induced by the last coordinate of $R_{\Gamma}(\Z/m\Z)$. We denote this semigroup by $R^{\pr}_{\Gamma}(\Z/m\Z)$.

\begin{theorem}\label{theorem:BZ_triangles_group_based_models_inclusion}
For $m\in \N_{\geq 2}$
\begin{equation}
R^{\pr}_{0,3}(\Z/m\Z) \subset \pr(\BZ(\SL_m(\C))).
\end{equation}
For $m\in \{2,3\}$ the inclusion is equality.  
\end{theorem}

By Proposition~\ref{ueno}, the set $\pr(\BZ(\SL_m(\C)))$ and the projection of $S_{0,3}(\SL_m(\C))$ that forgets the level are equal. Thus, the inclusion in Theorem~\ref{theorem:BZ_triangles_group_based_models_inclusion} follows from Theorem~\ref{theorem:conformal_blocks_group_based_models_inclusion}. Here we will give a combinatorial proof with the aim to provide a graphical tool for studying group-based models associated with groups $\Z/m\Z$.

\begin{proof}
We will prove that every vertex of $P_{0,3}(\Z/m\Z)\subset \R^{3(m-1)}$ corresponds to an element of $\pr(\BZ(\SL_m(\C)))\subset \R^{3(m-1)}$ (the unit vectors corresponding to the edges $e_1,e_2,e_3$ are replaced by the fundamental weights $\lambda,\mu,\nu$, respectively). This implies $ R^{\pr}_{0,3}(\Z/m\Z) \subseteq \pr(\BZ(\SL_m(\C)))$, since $R^{\pr}_{0,3}(\Z/m\Z)$ is generated by the vertices of $P_{0,3}(\Z/m\Z)$. 

\begin{itemize}
\item The zero vertex of $P_{0,3}(\Z/m\Z)$ equals the projection of the zero BZ triangle.

\item The vertex of $P_{0,3}(\Z/m\Z)$ with $x^{e_2}_{[i]}=x^{e_3}_{[m-i]}=1$ for fixed $i\in\{1,\ldots,m-1\}$  and all other coordinates $0$ is the projection of a BZ triangle with $1$'s at the vertices that lie on a line segment parallel to the northwest edge, and $0$'s on all the other vertices. Figure~\ref{figure:BZ} depicts honeycombs corresponding to all such BZ triangles for $\SL_4(\C)$.

\begin{figure}[ht]
\centering
\epsfxsize=240pt\epsfbox{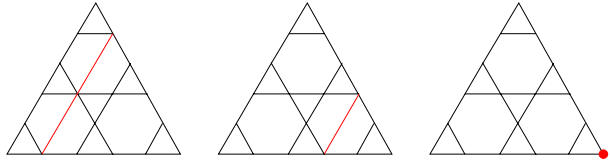}
\caption{BZ triangles with $1$'s at the vertices that lie on a line segment parallel to the northwest edge}
\label{figure:BZ}
\end{figure}

\item The vertex of $P_{0,3}(\Z/m\Z)$ with $x^{e_1}_{[i]}=x^{e_3}_{[m-i]}=1$ and all other coordinates $0$, and the vertex of $P_{0,3}(\Z/m\Z)$ with $x^{e_1}_{[i]}=x^{e_2}_{[m-i]}=1$  and all other coordinates $0$ are projections of BZ triangles with $1$'s at the vertices that lie on line segments parallel to the northeast and south edge, respectively.

\item The vertex of $P_{0,3}(\Z/m\Z)$ with $x^{e_1}_{[i]}=x^{e_2}_{[j]}=x^{e_3}_{[k]}=1$ for fixed $i,j,k\in \{1,\ldots ,m-1\}$ with $i+j+k=m$ and all other coordinates $0$ equals the projection of a BZ triangle with $1$'s at the vertices that lie on the three line segments that start at the middle point of a hexagon and go west, northeast and southeast. Figure~\ref{figure:BZ2} depicts honeycombs corresponding to all such BZ triangles for $\SL_4(\C)$.

\begin{figure}[ht]
\centering
\epsfxsize=240pt\epsfbox{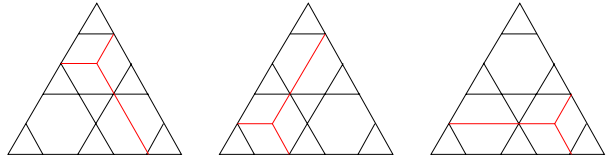}
\caption{BZ triangles with $1$'s at the vertices that lie on the three line segments that start at the middle point of a hexagon and go west, northeast and southeast}
\label{figure:BZ2}
\end{figure}

\item The vertex of $P_{0,3}(\Z/m\Z)$ with $x^{e_1}_{[i]}=x^{e_2}_{[j]}=x^{e_3}_{[k]}=1$ for fixed $i,j,k\in \{1,\ldots ,m-1\}$ with $i+j+k=2m$ and all other coordinates $0$ equals the projection of a BZ triangle with $1$'s at the vertices that lie on the three line segments that start at the middle point of a hexagon and go east, southwest and northwest. Figure~\ref{figure:BZ3} depicts honeycombs corresponding to all such BZ triangles for $\SL_4(\C)$. 
\end{itemize}

\begin{figure}[ht]
\centering
\epsfxsize=240pt\epsfbox{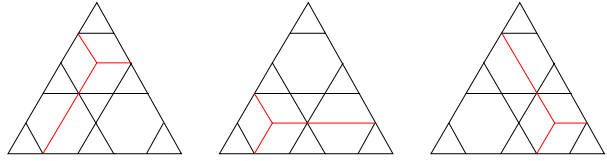}
\caption{BZ triangles with $1$'s on the vertices that lie on three line segments that start at the middle point of a hexagon and go east, southwest and northwest}
\label{figure:BZ3}
\end{figure}

Example~\ref{example:correspondence_for_r_equal_to_one_and_two} proves the equality for $m\in\{2,3\}$.
\end{proof}

The semigroups $\BZ(SL_2(\C))$ and $R^{\pr}_{0,3}(\Z/2\Z)$ are isomorphic.  The semigroups $\BZ(SL_3(\C))$ and $R^{\pr}_{0,3}(\Z/3\Z)$ are not isomorphic (but we have $\pr(\BZ(SL_3(\C)))=R^{\pr}_{0,3}(\Z/3\Z)$). The toric ideal corresponding to the semigroup $\BZ(SL_3(\C))$ is generated by one cubic polynomial and the toric ideal corresponding to $R^{pr}_{0,3}(\Z/3\Z)$ is generated by two cubic polynomials. For $m>3$, the inclusion in Theorem~\ref{theorem:BZ_triangles_group_based_models_inclusion} is strict. The $\SL_m(\C)$ BZ triangle given by 
\begin{align*}
&x(2m-3,0,2)=x(2m-2,1,0)=x(2m-4,2,1)=x(2m-6,4,1)=\ldots\\
&=x(4,2m-6,1)=x(2,2m-4,1)=x(1,2m-2,0)=x(0,2m-3,2)=1
\end{align*}
and all other coordinates $0$, see Figure~\ref{figure:BZNotPhyl} for the $\SL_4(\C)$ case,
has projection 
\begin{equation}
(1,0,\ldots,0,1;0,1,0,\ldots 0;0,\ldots 0,1,0),
\end{equation}
which does not belong to $R^{\pr}_{0,3}(\Z/m\Z)$.

\begin{figure}[ht]
\centering
\epsfxsize=80pt\epsfbox{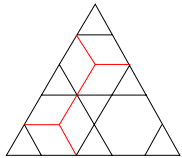}
\caption{A $\SL_4(\C)$ BZ triangle whose projection does not belong to  $R_{0,3}^{\pr}(\Z/ 4\Z)$}
\label{figure:BZNotPhyl}
\end{figure}

Let $\Gamma$ be a  trivalent graph. Define $\pr_{\Gamma}$ as the direct product of functions $\pr$ with one copy for each $v\in V(\Gamma)$.

\begin{corollary}\label{cor:connection_on_trees}
Let $\Gamma$ be a trivalent graph. For $m\in \N_{\geq 2}$
\begin{equation}
R^{\pr}_{\Gamma}(\Z/m\Z) \subset \pr_{\Gamma}(\BZ_{\Gamma}(\SL_m(\C))).
\end{equation}
Furthermore, for $m\in \{2,3\}$ the equality holds. 
\end{corollary}

\begin{proof}
By definition, we have
\begin{equation*}
 \pr_{\Gamma}(\BZ_{\Gamma}(\SL_m{\C}))=\prod_{v\in V(\Gamma)}\pr(\BZ(\SL_m(\C))^{(v)})\cap\bigcap _{e=(v_1,v_2)} \ker (\pr_e(x^{(v_1)})-\pr_e^*(x^{(v_2)})).
\end{equation*}
By Theorem~\ref{theorem:BZ_triangles_group_based_models_inclusion} and Lemma~\ref{lemma:group_based_semigroup}, this set contains $R_{\Gamma}(\Z/m\Z)$. 
\end{proof}

The inclusion in Corollary~\ref{cor:connection_on_trees} is strict for $m> 3$, since the inclusion in Theorem~\ref{theorem:BZ_triangles_group_based_models_inclusion} is strict for $m> 3$.

We now let $\BZ^{\gr}(\SL_3(\C)) \subset \BZ(\SL_3(\C))\times \Z_{\geq 0}$ be the sub-semigroup generated by the eight BZ triangles in Figure~\ref{figure:BZ_rank_two} and the zero BZ triangle, lifted to height $1$.  This semigroup is studied in \cite{depth_rule,Manon12}, in the context of conformal field theory.  We can now state the proof of Theorem~\ref{BZ}.

\begin{proof}[Proof of Theorem~\ref{BZ}]
 By Theorem~\ref{theorem:BZ_triangles_group_based_models_inclusion}, the generators of $\BZ^{\gr}(\SL_3(\C))$ are in one to one correspondence with the degree one elements of $R_{0,3}(\Z/3\Z)=\overline{R}_{0,3}(\Z/3\Z)$.  This defines a linear map $\phi :\BZ^{\gr}(\SL_3(\C)) \to \overline{R}_{0,3}(\Z/3\Z)$ given by passing to the boundaries of these elements.   The result for general $\Gamma$ then follows by modification of the argument used in Corollary \ref{cor:connection_on_trees}.
\end{proof}

The assignment of height $1$ to each of the generators of $\BZ(\SL_3(\C))$ is a result of the conformal field theory discussed in Section \ref{section:conformal_blocks_II}.   This is not carried out for general $m$, because a generating set for each $\BZ(\SL_m(\C)),$ and a way to assign heights to elements of this semigroup in a way consistent with the conformal field theory is still unknown.  This is related to the problem of constructing a toric degeneration of $V_{0, 3}(\SL_m(\C)).$

\bibliographystyle{alpha}
\bibliography{KMv4}

\end{document}